\documentclass[12pt]{amsart}
\usepackage{amsfonts}
\usepackage{amsfonts,latexsym,rawfonts,amsmath,amssymb,amsthm}
\usepackage[plainpages=false]{hyperref}
\usepackage{graphicx}

\numberwithin{equation}{section}

\RequirePackage{color}
\theoremstyle{plain}

\newtheorem{theorem}{Theorem}[section]

\newtheorem{definition}{Definition}[section]

\newtheorem{lemma}[theorem]{Lemma}
\newtheorem{proposition}[theorem]{Proposition}
\newtheorem{remark}[theorem]{Remark}

\newcommand{\beq}{\begin{equation}}
\newcommand{\eeq}{\end{equation}}
\newcommand{\beqs}{\begin{eqnarray*}}
\newcommand{\eeqs}{\end{eqnarray*}}
\newcommand{\beqn}{\begin{eqnarray}}
\newcommand{\eeqn}{\end{eqnarray}}
\newcommand{\beqa}{\begin{array}}
\newcommand{\eeqa}{\end{array}}

\def\phi{\varphi}

\begin{document}
\title{A Class of  Hessian quotient equations in Euclidean space}

\author{Xiaojuan Chen}
\address{Faculty of Mathematics and Statistics, Hubei Key Laboratory of Applied Mathematics, Hubei University,  Wuhan 430062, P.R. China}
\email{201911110410741@stu.hubu.edu.cn}

\author{Qiang Tu$^\ast$}
\address{Faculty of Mathematics and Statistics, Hubei Key Laboratory of Applied Mathematics, Hubei University,  Wuhan 430062, P.R. China}
\email{qiangtu@hubu.edu.cn}

\author{Ni  Xiang}
\address{Faculty of Mathematics and Statistics, Hubei Key Laboratory of Applied Mathematics, Hubei University,  Wuhan 430062, P.R. China}
\email{nixiang@hubu.edu.cn}



\thanks{This research was supported by funds from Hubei Provincial Department of Education
Key Projects D20181003 and the National Natural Science Foundation of China No.11971157.}
\thanks{$\ast$ Corresponding author}

\begin{abstract}
In this paper, we consider a Class of  Hessian quotient equations in Euclidean space. Under some sufficient condition, we obtain an
existence result by the standard degree theory based on the  prior
estimates for the solutions to the Hessian quotient equations.
\end{abstract}

\maketitle

\baselineskip18pt

\parskip3pt

 \section{Introduction}

Let $(M,g)$  be  a smooth, compact Riemannian manifold of dimension $n\geq 3$.
Define a  $(0,2)$ tensor $\eta$  on $M$ by
$$\eta_{ij}=Hg_{ij}-h_{ij},$$
where $g_{ij}$ and $h_{ij}$ are the first and second fundamental forms of $M$ respectively, $H(X)$ is the mean curvature at $X \in M$. In fact, $\eta$ is the first Newton transformation of $h$ with respect to $g$. The $\sigma_k$-curvature of $\eta$ is defined by
$$\sigma_k(\lambda(\eta)),$$
where $\sigma_k(\lambda(\eta))$ means $\sigma_k$ is applied to the eigenvalues of $g^{-1}\eta$ and the $k$-th elementary symmetric polynomial $\sigma_k$ is defined by:
\[\sigma_k(\lambda)=\sum_{1\le i_1<\cdots<i_k\le n}\lambda_{i_1}\cdots\lambda_{i_k}.\]
In this paper, we study the problem of prescribed Weingarten
curvature with $(k,l)$-Hessian quotient of $\lambda(\eta)$
\begin{eqnarray}\label{ht-Eq}
\frac{\sigma_k(\lambda(\eta))}{\sigma_l(\lambda(\eta))}=f(X, \nu), \quad 2\leq k\leq n,~ 0
\leq  l\leq k-2,
\end{eqnarray}
on a closed Riemannian manifold $M$, where $M$ is an embedded,
closed manifold in $\mathbb{R}^{n+1}$, $f$ is given smooth functions in $\mathbb{R}^{n+1}\times \mathbb{S}^n$.  $\nu(X)$ and $\kappa(X)=(\kappa_1(X), \cdots, \kappa_n(X))$  is the unit outer normal and the principal curbatures
of hypersurface at $X$. Note that
\begin{equation}
\lambda_i(\eta)=H-\kappa_i=\sum_{j\neq i} \kappa_j, \quad \quad \forall~ i =1,\cdots, n.
\end{equation}
To ensure the ellipticity of \eqref{ht-Eq}, we have to restrict the
class of  hypersurfaces.
\begin{definition}
A smooth hypersurface $M \subset \mathbb{R}^{n+1}$ is called
$(\eta, k)$-convex if $\lambda(\eta) \in \Gamma_k$ for any $X\in M$,
 where $\Gamma_k$ is the Garding's cone
\begin{eqnarray*}\label{cone}
\Gamma_{k}=\{\lambda \in \mathbb{R} ^n: \sigma_{j}(\lambda)>0, \forall ~ 1\leq j \leq k\}.
\end{eqnarray*}
\end{definition}
We mainly get the following theorem.

\begin{theorem}\label{Main}
Let $n\ge3, k\ge 2$, $0\leq l<k-1$ and $f\in C^2((\overline{B}_{r_2}\backslash B_{r_1}) \times \mathbb{S}^n)$ be a positive function. Assume that
\begin{eqnarray}\label{ASS1}
f(X, \frac{X}{|X|}) \leq \frac{C_n^k}{C_n^l}  (\frac{n-1}{r_2})^{k-l}
\quad \mbox{for} \quad |X|= r_2,
\end{eqnarray}
\begin{eqnarray}\label{ASS2}
f(X, \frac{X}{|X|}) \geq \frac{C_n^k}{C_n^l}  (\frac{n-1}{r_1})^{k-l} \quad \mbox{for} \quad |X|= r_1,
\end{eqnarray}
and
\begin{eqnarray}\label{ASS3}
\frac{\partial }{\partial \rho}\bigg[\rho^{k-l}f(X, \nu)\bigg]\leq
0 \quad \mbox{for} \quad r_1\leq |X|\leq r_2,
\end{eqnarray}
where $\rho=|X|$. Then there exists a $C^{4, \alpha}$,
$(\eta, k)$-convex, star-shaped and closed hypersurface $M$ in $\{r_1\leq
|X|\leq r_2\}$ satisfies equation (\ref{ht-Eq}) for any $\alpha\in (0,1)$.
\end{theorem}
\begin{remark}
The key to proving theorem \ref{Main} is to obtain the curvature estimate for
this  Hessian quotient type equation (\ref{ht-Eq}), which is established in Theorem 
 \ref{ht-C2e}. If we replace $\lambda(\eta)$ by $\kappa(X)$, Guan-Ren-Wang showed that the $C^2$ estimate fails for the quotient of curvature equation. 
\end{remark}

This kind of equations is motivated from the study of many important geometric
problems. For example, when $k=n$, \eqref{ht-Eq} becomes the following equation for $(\eta, n)$-convex hypersurface:
\begin{eqnarray}\label{ht-Eq-8}
\mbox{det} (\eta(X))=f(X, \nu),
\end{eqnarray}
which has been studied intensively by Sha \cite{Sha1, Sha2}, Wu\cite{Wu} and Harvey-Lawson \cite{HL2}. It is interesing to consider the curvature equation
\eqref{ht-Eq-8} and its generalization. In \cite{CJ}, Chu-Jiao
establish the curvature estimates  for  the equation, which
replace the left hand of \eqref{ht-Eq-8} by $\sigma_k(\eta(X))$.
It's worth noting that Theorem \ref{Main} recovers the existence results in \cite{CJ}. In the complex setting, when $k=n, l=0$, the equation \eqref{ht-Eq}  is called $(n-1)$ Monge-Amp$\grave{e}$re equation, which is related to the Gauduchon conjecture in complex geometry, more details see \cite{Ga}.

If we replace $\lambda(\eta)$ by $\kappa(X)$ and $l=0$ in \eqref{ht-Eq}, the equation \eqref{ht-Eq} becomes the classical prescribed curvature equation
\begin{eqnarray}\label{ht-Eq-9}
\sigma_k(\kappa(X))=f(X, \nu),
\end{eqnarray}
which has been widely studied in the past two decades.
In fact, the curvature estimates are the key part for this prescribed curvature euqation. When $k=n$, the curvature estimates are established by
Caffarelli-Nirenberg-Spruck \cite{Ca1}.
When $k=2$, the $C^2$ estimate for the equation \eqref{ht-Eq-9} was obtained by Guan-Ren-Wang \cite{Guan-Ren15}.  Spruck-Xiao \cite{Sp} extended $2$-convex case to space forms and give a simple proof for the Euclidean case. In\cite{Ren, Ren1}, Ren-Wang proved the $C^2$ estimate for $k=n-1$ and $n-2$
When $2<k<n$, $C^2$ estimate was also proved for equation of prescribing curvature
measures problem in \cite{Guan12, Guan09}, where $f(X, \nu) = \langle X, \nu \rangle \Tilde{f}(X)$.
Ivochkina \cite{Iv1,Iv2} considered the Dirichlet problem of the above equation  on domains in $\mathbb{R}^n$, and obtained $C^2$ estimates under some extra conditions on  the dependence of $f$ on $\nu$.
Caffarelli-Nirenberg-Spruck \cite{Ca} and Guan-Guan \cite{Guan02} proved the $C^2$ estimate if $f$ is independent of $\nu$ and depends only on $\nu$, respepectively. Moreover, Some results have been obtained by Li-Oliker \cite{Li-Ol} on unit sphere, Barbosa-de Lira-Oliker \cite{Ba-Li} on space forms, Jin-Li \cite{Jin} on hyperbolic space, Andrade-Barbosa-de Lira \cite{An} on warped product manifolds.

The organization of the paper is as follows.
In Sect. 2
we start with some preliminaries.
$C^0$, $C^1$ and $C^2$ estimates are given in Sect. 3.
In Sect. 4 we prove theorem \ref{Main}.


\section{Preliminaries}

\subsection{Setting and General facts}
For later convenience, we first state our conventions on Riemann
Curvature tensor and derivative notation. Let $M$ be a smooth
manifold and $g$ be a Riemannian metric on $M$ with Levi-Civita
connection $\nabla$. For a $(s, r)$-tensor field $\alpha$ on $M$,
its covariant derivative $\nabla \alpha$ is a $(s, r+1)$-tensor
field given by
\begin{eqnarray*}
&&\nabla \alpha(Y^1, .., Y^s, X_1, ..., X_r, X)
\\&=&\nabla_{X} \alpha(Y^1, .., Y^s, X_1, ..., X_r)\\&=&X(\alpha(Y^1, .., Y^s, X_1, ..., X_r))-
\alpha(\nabla_X Y^1, .., Y^s, X_1, ..., X_r)\\&&-...-\alpha(Y^1, ..,
Y^s, X_1, ..., \nabla_X  X_r).
\end{eqnarray*}
The coordinate expression of which is denoted by
$$\nabla \alpha=(\alpha_{k_{1}\cdot\cdot\cdot
k_{r}; k_{r+1}}^{l_{1}\cdot\cdot\cdot l_{s}}).$$ We can continue to
define the second covariant derivative of $\alpha$ as follows:
\begin{eqnarray*}
&&\nabla^2 \alpha(Y^1, .., Y^s, X_1, ..., X_r, X, Y)
=(\nabla_{Y}(\nabla\alpha))(Y^1, .., Y^s, X_1, ..., X_r, X).
\end{eqnarray*}
The coordinate expression of which is denoted by
$$\nabla^2 \alpha=(\alpha_{k_{1}\cdot\cdot\cdot
k_{r}; k_{r+1}k_{r+2}}^{l_{1}\cdot\cdot\cdot l_{s}}).$$ Similarly,
we can also define the higher order covariant derivative of
$\alpha$:
$$\nabla^3 \alpha=\nabla(\nabla^2 \alpha), \nabla^4 \alpha=\nabla(\nabla^3 \alpha), ... ,$$
and so on. For simplicity, the coordinate expression of the
covariant differentiation will usually be denoted by indices without
semicolons, e.g.,
$$u_{i}, \quad u_{ij} \quad \mbox{or} \quad
u_{ijk}$$
 for a function $u: M\rightarrow \mathbb{R}$.

Our convention for the Riemannian curvature $(3,1)$-tensor Rm is
defined by
\begin{equation*}
Rm(X, Y)Z=-\nabla_{X}\nabla_{Y}Z+\nabla_{Y}\nabla_{X}Z+\nabla_{[X,
Y]}Z.
\end{equation*}
Pick a local coordinate chart $\{x^i\}_{i=1}^{n}$ of $M$. The
component of the $(3,1)$-tensor $Rm$ is defined by
\begin{equation*}
Rm\bigg({\frac{\partial}{\partial x^i}}, {\frac{\partial}{\partial
x^j}}\bigg){\frac{\partial}{\partial x^k}}=R_{ijk}^{\ \ \
l}{\frac{\partial}{\partial x^l}}
\end{equation*}
and $R_{ijkl}=g_{lm}R_{ijk}^{\ \ \ m}$. Then, we have the
standard commutation formulas (Ricci identities):
\begin{eqnarray*}\label{RI}
\alpha_{k_{1}\cdot\cdot\cdot k_{r};\ j i}^{l_{1}\cdot\cdot\cdot
l_{s}}-\alpha_{k_{1}\cdot\cdot\cdot k_{r};\ i
j}^{l_{1}\cdot\cdot\cdot l_{s}}=\sum_{a=1}^{r}R^{\ \ \ m}_{ijk_{l}}
\alpha_{k_{1}\cdot\cdot\cdot k_{a-1}m k_{a+1}\cdot\cdot\cdot
k_{r}}^{l_{1}\cdot\cdot\cdot l_{s}}-\sum_{b=1}^{s}R^{\ \ \
l_b}_{ijm} \alpha_{k_{1}\cdot\cdot\cdot k_{r}}^{l_{1}\cdot\cdot\cdot
l_{b-1}m l_{b+1}\cdot\cdot\cdot l_{r}}.
\end{eqnarray*}

Let $M$ be an immersed hypersurface in $\mathbb{R}^{n+1}$. Denote
$R_{ijkl}$ to be the Riemannian curvature of $M\subset
\mathbb{R}^{n+1}$ with the induced metric $g$. Pick a local
coordinate chart $\{x^i\}_{i=1}^{n}$ on $M$. Let $\nu$ be a given
unit normal and $h_{ij}$ be the second fundamental form $A$ of the
hypersurface with respect to $\nu$, that is
$$h_{ij}=-\langle\frac{\partial^2 X}{\partial x^i\partial x^j}, \nu\rangle_{\mathbb{R}^{n+1}}.$$
Recalling the following identities
\begin{equation}\label{Gauss for}
\nabla_i \nabla_j X=-h_{ij}\nu, \quad \quad \mbox{Gauss formula}
\end{equation}

\begin{equation}\label{Wein for}
\nabla_i \nu=h_{ij}X^j, \quad \quad \mbox{Weingarten formula}
\end{equation}

\begin{equation*}\label{Gauss}
R_{ijkl}=h_{ik}h_{jl}-h_{il}h_{jk}, \quad \quad \mbox{Gauss
equation}
\end{equation*}

\begin{equation}\label{Codazzi}
\nabla_{k}h_{ij}=\nabla_{j}h_{ik}, \quad \quad \mbox{Codazzi
equation}
\end{equation}
where $X^j=g^{ik}\nabla_i X$.
Moreover, we have
\begin{eqnarray}\label{2rd}
\nabla_{i}\nabla_{j}h_{kl}
&=&\nabla_{k}\nabla_{l}h_{ij}+h^{m}_{j}(h_{il}h_{km}-h_{im}h_{kl})+h^{m}_{l}(h_{ij}h_{km}-h_{im}h_{kj}).
\end{eqnarray}

\subsection{Star-shaped hypersurfaces in $\mathbb{R}^{n+1}$}

Let $M$ be a star-shaped hypersurface in $\mathbb{R}^{n+1}$ which can represented by
\begin{eqnarray*}
X(x)=\rho(x)x, \quad \mbox{for} \quad x \in \mathbb{S}^n,
\end{eqnarray*}
where $X$ is the position vector of the hypersurface $M$ in $\mathbb{R}^{n+1}$.

Let $\{e_1,...,e_n\}$ be a smooth local orthonormal frame field on
$\mathbb{S}^n$ and $e_\rho$ be the radial vector field
in $\mathbb{R}^{n+1}$. $D_i\rho=D_{e_i} \rho$, $D_iD_j\rho=D^2
\rho(e_i, e_j)$ denote the covariant derivatives of $u$ with respect
to the round metric $\sigma$ of $\mathbb{S}^n$. Then, the following formulas
hold:

(i) The tangential vector on $M$ is
\begin{eqnarray*}
X_{i}=\rho e_{i}+D_i\rho e_{\rho}
\end{eqnarray*}
and the corresponding outward unit normal vector is given by
\begin{eqnarray}\label{Nor}
\nu=\frac{1}{v}\left(e_\rho-\frac{1}{\rho^2} D^j\rho e_j\right),
\end{eqnarray}
where $v=\sqrt{1+\rho^{-2}|D \rho|^2}$ with $D^j \rho=\sigma^{ij}D_i\rho$.

(ii) The induced metric $g$ on $M$ has the form
\begin{equation*}
g_{ij}=\rho^2\sigma_{ij}+D_i\rho D_j\rho
\end{equation*}
and its inverse is given by
\begin{equation*}
g^{ij}=\frac{1}{\rho^2}\left(\sigma^{ij}-\frac{D^i\rho D^j\rho}{\rho^2
v^{2}}\right).
\end{equation*}

(iii) The second fundamental form of $M$ is given by
\begin{eqnarray*}
h_{ij}=\frac{1}{v}\left(-D_iD_j\rho+\rho
\sigma_{ij}+\frac{2}{\rho}D_i\rho D_j\rho\right)
\end{eqnarray*}
and
\begin{eqnarray}\label{h_ij}
h^{i}_{j}=\frac{1}{\rho
v}\left(\delta^{i}_{j}+[-\sigma^{ik}+\frac{D^i\rho
D^k\rho}{\rho^2v}]D_jD_k(\log \rho)\right).
\end{eqnarray}

The following Newton-Maclaurin inequality (see \cite{Tr90, LT94}) will be used frequently.
\begin{lemma} \label{lemma1}
\textit{\ Let} $\lambda\in\mathbb{R}^n$. \textit{ For }$0\leq
l<k\leq n,$ $r>s\ge0, k\ge r, l\ge s$, \textit{\ the following is
the Newton-Maclaurin inequality }

(1)
\[
k(n-l+1)\sigma_{l-1}(\lambda)\sigma_{k}(\lambda)\leq
l(n-k+1)\sigma_{l}(\lambda)\sigma_{k-1}(\lambda).
\]

(2)
\[\big[\frac{\sigma_k(\lambda)/C^k_n}{\sigma_l(\lambda)/C^l_n}\big]^\frac{1}{k-l}
\le\big[\frac{\sigma_r(\lambda)/C^r_n}{\sigma_s(\lambda)/C^s_n}\big]^{\frac{1}{r-s}}
\quad\textit{for }  \lambda\in \Gamma_k.\]
\end{lemma}

For convenience, we introduce the following notations:
\begin{eqnarray}
G(\eta):= \left(\frac{\sigma_k(\eta)}{\sigma_l(\eta)}\right)^{\frac{1}{k-l}},\quad G^{ij}:=\frac{\partial G}{ \partial \eta_{ij}}, \quad
G^{ij, rs}:= \frac{\partial^2 G}{\partial \eta_{ij} \partial \eta_{rs}}, F^{ii}:=\sum_{k\neq i} G^{kk}.
\end{eqnarray}
Thus,
$$G^{ii}= \frac{1}{k-l} \left(\frac{\sigma_k(\eta)}{\sigma_l(\eta)}\right)^{\frac{1}{k-l}-1} \frac{\sigma_{k-1}(\eta| i)\sigma_l(\eta)-\sigma_k(\eta)\sigma_{l-1}(\eta| i)}{\sigma_l^2(\eta)}.$$
If $\eta=\mbox{diag}(\mu_1, \mu_2, \cdots, \mu_n)$ with $\mu_1 \leq \mu_2\leq \cdots \leq \mu_n$. It follows that
$$G^{11}\geq G^{22} \geq \cdots \geq G^{nn}, \quad F^{11} \leq F^{22} \leq \cdots \leq F^{nn}.$$

To handle the ellipticity of the equation \eqref{ht-Eq}, we need the
following important proposition and its proof is the same as
Proposition 2.2.3 in \cite{CQ1}.
\begin{proposition}\label{ellipticconcave}
Let $M$ be a smooth $(\eta, k)$-convex closed hypersurface in $\mathbb{R}^{n+1}$
and $0\leq l< k-1$. Then the operator
\begin{eqnarray}
G(\eta_{ij}(X))=\left(\frac{\sigma_k(\lambda(\eta))}{\sigma_{l}(\lambda(\eta))}\right)^{\frac{1}{k-l}}
\end{eqnarray}
is elliptic and concave with respect to $\eta_{ij}(X)$. Moreover we have
\begin{eqnarray}
\sum G^{ii} \geq \left(\frac{C_n^k}{C_n^l}\right)^{\frac{1}{k-l}}.
\end{eqnarray}
\end{proposition}

\begin{proposition}\label{th-lem-07}
Let $\eta$ be a diagonal matrix with $\lambda(\eta)\in \Gamma_k$, $0\leq l \leq k-2$ and $k\geq 3$. Then
\begin{eqnarray}
-G^{1i, i1}(\eta)=\frac{G^{11}-G^{ii}}{\eta_{ii}-\eta_{11}}
\end{eqnarray}
for $i\geq 2$.
\end{proposition}
\begin{proof}
We only need to proof the statement in case $l\geq 1$. According to the Proposition 2.1.4 in \cite{CQ1}, we know that
\begin{eqnarray*}
\frac{\partial \sigma_k(\eta)}{\partial \eta_{ij}}=\begin{cases}  \sigma_{k-1}(\eta|i), \quad ~\mbox{if}~i=j,\\
0,\quad \quad \quad~\mbox{if} ~i\neq j,
\end{cases}
\end{eqnarray*}
and
\begin{eqnarray*}
\frac{\partial^2 \sigma_k(\eta)}{\partial \eta_{ij} \partial \eta_{kl}}=\begin{cases}  \sigma_{k-2}(\eta|ik), \quad ~\mbox{if}~i=j, k=l, i \neq k,\\
-\sigma_{k-2} (\eta|ik),  ~\mbox{if}~ i=l, j=k, i\neq j,\\
0,\quad \quad \quad \quad\quad \quad ~\mbox{otherwise}.
\end{cases}
\end{eqnarray*}
Thus
\begin{eqnarray*}
G^{1i, i1}(\eta) (\eta_{ii}-\eta_{11})&=& \frac{1}{k-l} \left(\frac{\sigma_k(\eta)}{\sigma_l(\eta)}\right)^{\frac{1}{k-l}-1} \frac{1}{\sigma_l^2(\eta)} \left( \frac{\partial^2 \sigma_k(\eta)}{\partial \eta_{1i} \partial \eta_{i1}} \sigma_l(\eta) -\frac{\partial^2 \sigma_l(\eta)}{\partial \eta_{1i} \partial \eta_{i1}} \sigma_k(\eta) \right)(\eta_{ii}-\eta_{11})\\
&=&-\frac{1}{k-l} \left(\frac{\sigma_k(\eta)}{\sigma_l(\eta)}\right)^{\frac{1}{k-l}-1} \frac{1}{\sigma_l^2(\eta)} \left( \sigma_{k-2} (\eta|1i) \sigma_l(\eta)- \sigma_{l-2}(\eta|1i) \sigma_k(\eta) \right) (\eta_{ii}-\eta_{11}).
\end{eqnarray*}
Note that $\sigma_{k-1} (\eta|1)=\sigma_{k-1}(\eta|1i)+\eta_{ii} \sigma_{k-2}(\eta|1i)$, Then
\begin{eqnarray*}
G^{1i, i1}(\eta) (\eta_{ii}-\eta_{11})&=& -\frac{1}{k-l} \left(\frac{\sigma_k(\eta)}{\sigma_l(\eta)}\right)^{\frac{1}{k-l}-1} \frac{1}{\sigma_l^2(\eta)} \left(\sigma_{k-1} (\eta|1) \sigma_l(\eta)- \sigma_{l-1}(\eta|1) \sigma_k(\eta)  \right)\\
&&-\frac{1}{k-l} \left(\frac{\sigma_k(\eta)}{\sigma_l(\eta)}\right)^{\frac{1}{k-l}-1} \frac{1}{\sigma_l^2(\eta)} \left(\sigma_{k-1} (\eta|i) \sigma_l(\eta)- \sigma_{l-1}(\eta|i) \sigma_k(\eta)  \right)\\
&=&G^{ii}-G^{11}.
\end{eqnarray*}
\end{proof}

\section{The  prior estimates}

In order to prove Theorem \ref{Main}, we consider the family of
equations as in \cite{An, Li-Sh} for $0\leq t\leq 1$
\begin{eqnarray}\label{Eq2}
\frac{\sigma_k(\lambda(\eta))}{\sigma_l(\lambda(\eta))}=f^t(X, \nu),
\end{eqnarray}
where
\begin{eqnarray*}
f^t(X, \nu)=tf(X, \nu)+ (1-t)  \frac{C_n^k}{C_n^l}  (n-1)^{k-l} \left( \frac{1}{|X|^{k-l}} +\epsilon (\frac{1}{|X|^{k-l}}-1)\right),
\end{eqnarray*}
where the constant $\epsilon$ is small sufficiently such that
\begin{eqnarray*}
\min_{r_1\leq \rho\leq r_2}\left( \frac{1}{\rho^{k-l}} +\epsilon (\frac{1}{\rho^{k-l}}-1)\right)\geq c_0>0
\end{eqnarray*}
for some positive constant $c_0$.

\subsection{$C^0$ Estimates}

Now, we can prove the following proposition which asserts that  the
solution of the equation \eqref{ht-Eq} have uniform $C^0$ bound.

\begin{theorem}\label{ht-C0}
Assume  $f\in C^2((\overline{B}_{r_2}\backslash B_{r_1}) \times \mathbb{S}^n)$ is a positive function.
Under the assumptions \eqref{ASS1} and \eqref{ASS2} mentioned in
Theorem \ref{Main} , if $M\subset \mathbb{R}^{n+1}$ is a
star-shaped,  $(\eta, k)$-convex hypersurface satisfied the equation
\eqref{Eq2} for a given $t \in [0, 1]$, then
\begin{eqnarray*}
r_1<\rho(X)<r_2, \quad \forall \ X \in M.
\end{eqnarray*}
\end{theorem}

\begin{proof}
Assume $\rho(x)$ attains its maximum at $x_0 \in \mathbb{S}^n$ and
$\rho(x_0)\geq r_2$, then recalling \eqref{h_ij}
\begin{eqnarray*}
h^{i}_{j}=\frac{1}{\rho v}\left(\delta^{i}_{j}+[-\sigma^{im}+\frac{D^i\rho
D^m\rho}{\rho^2v}]D_jD_m(\log \rho)\right),
\end{eqnarray*}
which implies
\begin{eqnarray*}
h^{i}_{j}(x_0)=\frac{1}{\rho}[\delta^{i}_{j}-\sigma^{im}D_jD_m(\log
\rho)]\geq\frac{1}{\rho}\delta^{i}_{j}.
\end{eqnarray*}
Then
\begin{eqnarray*}
\eta^{i}_{j}(x_0)=H\delta_j^i-h^i_j\geq \frac{n-1}{\rho}\delta_j^i.
\end{eqnarray*}
Note that $\frac{\sigma_{k}}{\sigma_{l}}$ for $0\leq l\leq k-2$ is concave in
$\Gamma_{k}$. Thus
\begin{eqnarray*}
\frac{\sigma_k(\lambda(\eta))}{\sigma_{l}(\lambda(\eta))} \geq
\frac{\sigma_k(\frac{n-1}{\rho}\delta^{i}_{j})}{\sigma_{l}(\frac{n-1}{\rho}\delta^{i}_{j})}=\frac{C_n^k}{C_n^l}  (\frac{n-1}{\rho})^{k-l}.
\end{eqnarray*}
On the other hand at $x_0$, the unit outer normal $\nu$ is parallel to $M$, i.e., $\nu= \frac{X}{|X|}$. If $\rho(x_0)=r_2$, then
\begin{eqnarray*}
\frac{C_n^k}{C_n^l}  (\frac{n-1}{r_2})^{k-l}>f^t(X, \nu)=\frac{\sigma_k(\lambda(\eta))}{\sigma_{l}(\lambda(\eta))} \geq \frac{C_n^k}{C_n^l}  (\frac{n-1}{r_2})^{k-l}
\end{eqnarray*}
which is a contradiction. This shows $\sup \rho \leq r_2$. Similarly, we get
$\inf_M \rho \geq r_1$
in view of \eqref{ASS2}.
\end{proof}

Now, we prove the following uniqueness result.

\begin{proposition}\label{Uni}
For $t=0$, there exists an unique admissible solution of the
equation \eqref{Eq2}, namely $M=\mathbb{S}^n$.
\end{proposition}

\begin{proof}
Let $X$ be a solution of \eqref{Eq2},  for $t=0$
\begin{eqnarray*}
\frac{\sigma_k(\lambda(\eta))}{\sigma_l(\lambda(\eta))}=\frac{C_n^k}{C_n^l}  (n-1)^{k-l} \left( \frac{1}{|X|^{k-l}} +\epsilon (\frac{1}{|X|^{k-l}}-1)\right).
\end{eqnarray*}

Assume $\rho(x)$ attains its maximum $\rho_{max}$ at $x_0 \in
\mathbb{S}^n$, then
\begin{eqnarray*}
\frac{\sigma_k(\lambda(\eta))}{\sigma_{l}(\lambda(\eta))} \geq
\frac{C_n^k}{C_n^l}  (\frac{n-1}{\rho})^{k-l}
\end{eqnarray*}
which implies
\begin{eqnarray*}
\rho_{max}\leq 1.
\end{eqnarray*}
Similarly,
\begin{eqnarray*}
\rho_{min}\geq 1.
\end{eqnarray*}
Thus, $\rho=1$ is the unique solution of \eqref{Eq2} for $t=0$.
\end{proof}

\subsection{$C^1$ Estimates}

In this section, we establish the gradient estimate for the
equation. The treatment of this section follows largely from Lemma 4.1 of \cite{CJ}.

Recalling that a star-shaped hypersurface $M$ in $\mathbb{R}^{n+1}$ can be represented by
\begin{eqnarray*}
X(x)=\rho(x)x \quad \mbox{for} \quad x \in \mathbb{S}^n,
\end{eqnarray*}
where $X$ is the position vector of the hypersurface $M$ in $\mathbb{R}^{n+1}$.

In order to get the gradient estimate, we define a funcion $u=\langle X, \nu\rangle$. It is clear that
$$u=\frac{\rho^2}{\sqrt{\rho^{2}+|D \rho|^2}}.$$

\begin{theorem}\label{ht-C1e}
Under the assumption \eqref{ASS3}, if the closed star-shaped  $(\eta, k)$-convex hypersurface $M$ satisfying the curvature equation \eqref{ht-Eq}
 and the $\rho$ has positive upper and lower
bound. Then there exists a constant C depending only on $n, k, l,  \inf \rho,  \sup \rho, \inf f$ and $\|f\|_{C^1}$ such that
\begin{equation*}
|D \rho|\leq C.
\end{equation*}
\end{theorem}

\begin{proof}
It is sufficient to obtain a positive lower bound of $\langle X,
\nu\rangle$. We consider
\begin{eqnarray*}
\phi=-\log u+\gamma(|X|^2),
\end{eqnarray*}
where $\gamma(t)$ is a function which will be chosen later.
Assume $X_0$ is the maximum value point of $\phi$. If $X$ is
parallel to the normal direction $\nu$ of at $X_0$, we have
$\langle X, \nu\rangle=|X|^2$. Thus, our result holds.
So, we assume $X$ is not
parallel to the normal direction $\nu$ at $X_0$, we may choose the
local orthonormal frame $\{e_1, \cdots, e_n\}$ on $M$ satisfying
\begin{eqnarray*}
\langle X, e_1\rangle\neq 0, \quad \mbox{and} \quad \langle X,
e_i\rangle=0, \quad  \forall ~ i\geq 2.
\end{eqnarray*}
Using Weingarten equation, we obtain
$$u_i=\sum_j h_{ij} \langle X, e_j\rangle=h_{i1} \langle X, e_1\rangle.$$
Then, we arrive at $X_0$,
\begin{eqnarray}\label{Par-1}
0=\phi_i= - \frac{u_i}{u}+ 2\gamma^{\prime}\langle
X, e_i\rangle=- \frac{h_{i1} \langle
X, e_1\rangle}{u}+2\gamma^{\prime}\langle
X, e_i\rangle,
\end{eqnarray}
which implies that
\begin{eqnarray}\label{Par-1-1}
h_{11}=2u\gamma^{\prime}, \quad h_{1i}=0, \quad \forall ~
i\geq 2.
\end{eqnarray}
Therefore, we can rotate the coordinate system such that
$\{e_i\}_{i=1}^{n}$ are the principal curvature directions of the
second fundamental form $(h_{ij})$, i.e., $h_{ij}=h_{ii}\delta_{ij}$.
Thus,
\begin{eqnarray}\label{ht-c1-02}
0&\geq& F^{ii}\phi_{ii}\\
\nonumber&=&F^{ii} \left(-\frac{u_{ii}}{u}+\frac{u_{i}^2}{u^2}+\gamma^{\prime
\prime}(|X|^{2})_i^2+\gamma^{\prime}(|X|^2)_{ii} \right)
\\\nonumber&=&-\frac{F^{ii}u_{ii}}{u}+4((\gamma^{\prime})^2+ \gamma^{\prime\prime}) F^{11} \langle X, e_1\rangle^2+ \gamma^{\prime}F^{ii}(|X|^2)_{ii}.
\end{eqnarray}

Since $\eta_{ii}= \sum_{j\neq i} h_{jj}$, we have
$$\sum_i \eta_{ii} = (n-1) \sum_{i} h_{ii}, \quad h_{ii}= \frac{1}{n-1} \sum_k \eta_{kk}- \eta_{ii}.$$
It follows that
\begin{equation}\label{ht-c2-03}
\begin{aligned}
\sum_i F^{ii} h_{ii} =&\sum_i \left( \sum_k G^{kk} -G^{ii}\right) \left(\frac{1}{n-1} \sum_l \eta_{ll} -\eta_{ii}\right)\\
=& \sum_i G^{ii} \eta_{ii}\\
=&  \frac{1}{k-l} \left(\frac{\sigma_k(\eta)}{\sigma_l(\eta)}\right)^{\frac{1}{k-l}-1} \frac{\sum_i \eta_{ii}\sigma_{k-1}(\eta| i)\sigma_l(\eta)-\sigma_k(\eta) \sum_i \eta_{ii}\sigma_{l-1}(\eta| i)}{\sigma_l^2(\eta)}\\
=& \tilde{f},
\end{aligned}
\end{equation}
where $\tilde{f}=f^{\frac{1}{k-l}}$. Combining with \eqref{Gauss for}, we have
\begin{eqnarray}\label{ht-c1-04}
\gamma^{\prime}F^{ii}(|X|^2)_{ii}=2\gamma^{\prime}\sum_i F^{ii}-2\gamma^{\prime}u \tilde{f}.
\end{eqnarray}
Note that the curvature equation \eqref{ht-Eq}  can be written as
\begin{equation}\label{ht-eq-2}
G(\eta)=\tilde{f},
\end{equation}
 Differentiating \eqref{ht-eq-2}, we obtian
\begin{eqnarray*}
G^{ii} \eta_{iik}= (d_X \tilde{f})(e_k)+h_{kk} (d_{\nu} \tilde{f})(e_k).
\end{eqnarray*}
In fact
\begin{equation}\begin{aligned}\label{ht-c2-110}
F^{ii}h_{iik}&=\sum_i\left (\sum_jG^{jj}-G^{ii}\right )h_{iik}\\
&=\left(\sum_jG^{jj} \right)H_{k}-\sum_iG^{ii}h_{iik}\\
&=\sum_iG^{ii}\eta_{iik}\\
&=(d_X \tilde{f})(e_k)+h_{kk} (d_{\nu} \tilde{f})(e_k).
\end{aligned}
\end{equation}
Using Gauss formula \eqref{Gauss for},  Weingarten formula \eqref{Wein for} and Codazzi formula \eqref{Codazzi}
$$u_i=h_{ii}\langle X,e_i\rangle, \quad u_{ii}=\sum_k h_{iik}\langle X, e_k\rangle-uh_{ii}^2+h_{ii}.$$
Then
\begin{eqnarray}\label{ht-c1-03}
-\frac{F^{ii}u_{ii}}{u}&=&-\frac{\langle X, e_1\rangle}{u} \left((d_X \tilde{f})(e_1)+h_{11} (d_{\nu} \tilde{f})(e_1)\right) + F^{ii}h_{ii}^2 - \frac{\tilde{f}}{u}.
\end{eqnarray}


 Substituting  \eqref{ht-c1-03} and \eqref{ht-c1-04} into \eqref{ht-c1-02},
\begin{eqnarray}\label{ht-c1-05}
0&\geq&-\frac{\langle X, e_1\rangle}{u} \left((d_X \tilde{f})(e_1)+h_{11} (d_{\nu} \tilde{f})(e_1)\right) + F^{ii}h_{ii}^2 - \frac{\tilde{f}}{u}\\
\nonumber&&+4((\gamma^{\prime})^2+ \gamma^{\prime\prime}) F^{11} \langle X, e_1\rangle^2
+2\gamma^{\prime}\sum_i F^{ii}-2\gamma^{\prime}u \tilde{f}\\
\nonumber&=&- \frac{1}{u}\left(\langle X, e_1\rangle (d_X \tilde{f})(e_1) +  \tilde{f}\right)-2\gamma^{\prime} \langle X, e_1\rangle (d_{\nu} \tilde{f})(e_1)\\
 \nonumber&&+F^{ii}h_{ii}^2+4((\gamma^{\prime})^2+ \gamma^{\prime\prime}) F^{11} \langle X, e_1\rangle^2
+2\gamma^{\prime}\sum_i F^{ii}-2\gamma^{\prime}u \tilde{f}.
\end{eqnarray}
Since $X= \langle X, e_1\rangle e_1+ \langle X, \nu\rangle \nu$,
$$(d_X \tilde{f})(X)=\langle X, e_1 \rangle  (d_X \tilde{f})(e_1) + \langle X, \nu \rangle (d_{\nu} \tilde{f})(\nu).$$
From \eqref{ASS3} and $X=\rho(x) x$, we see that
\begin{eqnarray*}
0&\geq& \frac{\partial}{\partial \rho} \left( \rho^{k-l} f \right)=\frac{\partial}{\partial \rho} \left( \rho^{k-l} \tilde{f}^{k-l} \right)\\
&=&(k-l) (\rho \tilde{f})^{k-l} (\tilde{f}+ (d_X \tilde{f})(X))\\
&=&(k-l) (\rho \tilde{f})^{k-l} \left(\tilde{f}+
\langle X, e_1 \rangle  (d_X \tilde{f})(e_1) + \langle X, \nu \rangle (d_{\nu} \tilde{f})(\nu)
\right).
\end{eqnarray*}
It follows that
$$-\left(\tilde{f}+
\langle X, e_1 \rangle  (d_X \tilde{f})(e_1) \right)\geq  u (d_{\nu} \tilde{f})(\nu),$$
which implies
\begin{eqnarray}\label{ht-c1-06}
0&\geq&(d_{\nu} \tilde{f})(\nu)-2\gamma^{\prime} \langle X, e_1\rangle (d_{\nu} \tilde{f})(e_1)+F^{ii}h_{ii}^2\\
 \nonumber&&+4((\gamma^{\prime})^2+ \gamma^{\prime\prime}) F^{11} \langle X, e_1\rangle^2
+2\gamma^{\prime}\sum_i F^{ii}-2\gamma^{\prime}u \tilde{f}.
\end{eqnarray}
Choosing
$$\gamma(t) =\frac{\beta}{t},$$
where $\beta$ is a constant to be determined later. Recalling that $h_{11}=2\gamma^{\prime}u<0$ at $X_0$. From $H>0$ we know that
$$F^{11}=\sum_{j \neq 1} G^{jj} \geq \frac{1}{2} \sum_i G^{ii} =\frac{1}{2(n-1)}\sum_iF^{ii}\geq \frac{1}{2}\left(\frac{C_n^k}{C_n^l}\right)^{\frac{1}{k-l}}.$$
Substituting these into \eqref{ht-c1-06},
\begin{eqnarray*}
0&\geq&-C\sum_i F^{ii}-2C|\gamma^{\prime}| |\langle X, e_1\rangle| \sum_i F^{ii}+\frac{4}{n-1}((\gamma^{\prime})^2+ \gamma^{\prime\prime})  \langle X, e_1\rangle^2 \sum_i F^{ii}
+2\gamma^{\prime}\sum_i F^{ii},
\end{eqnarray*}
which implies
\begin{eqnarray}\label{ht-c1-07}
0&\geq&\frac{4}{n-1}(\frac{\beta^2}{\rho^8}+ \frac{4\beta^2}{\rho^{12}})  \langle X, e_1\rangle^2  -2C\frac{\beta}{\rho^4} |\langle X, e_1\rangle|
-2\frac{\beta}{\rho^4}-C \\
\nonumber&\geq&\frac{4}{n-1}(\frac{1}{\rho_2^8}+ \frac{4}{\rho_2^{12}})\beta^2  \langle X, e_1\rangle^2  -2C\frac{\beta}{\rho_1^4} |\langle X, e_1\rangle|
-2\frac{\beta}{\rho_1^4}-C,
\end{eqnarray}
where $\rho_1=\inf_M \rho, \rho_2=\sup_M \rho$.  So we can choose $\beta$ sufficiently large such that
$$|\langle X, e_1\rangle| < \frac{1}{2} \inf_M \rho,$$
combining with the fact $\rho^2=u^2+|\langle X, e_1\rangle|^2$, we obtain
$$u(X_0) \geq C.$$
So our proof is completed.
\end{proof}

\subsection{$C^2$ Estimates}

  Under the  the assumption \eqref{ASS1}, \eqref{ASS2}  and  \eqref{ASS3},  from Theorem   \ref{ht-C0} and  \ref{ht-C1e} we know that
there exists a positive constant $C$ depending on $\inf_{M} \rho$ and $\|\rho\|_{C^1}$ such that
$$\frac{1}{C} \leq inf_{M} u \leq u\leq \sup_{M} u \leq C.$$

\begin{theorem}\label{ht-C2e}
Let $M$ be a closed star-shaped  $(\eta, k)$-convex hypersurface satisfying the curvature equation \eqref{ht-Eq} for some positive function $f\in C^2(\Gamma)$, where $\Gamma$ is an open neighborhood of the unit normal boundle of $M$ in $\mathbb{R}^{n+1} \times \mathbb{S}^n$. Then, there exists a constant C depending only on $n, k, \|\rho\|_{C^1}, \inf_{M} \rho, \inf f$ and $\|f\|_{C^2}$ such that for $1\leq i\leq n$
\begin{equation*}
|\kappa_{i}(X)|\le C, \quad \forall ~ X \in M.
\end{equation*}
\end{theorem}

\begin{proof}
Since $\eta\in\Gamma_{k}\subset\Gamma_{1}$, we see that the mean curvature is positive. It suffices to prove that the largest curvature $\kappa_{\mbox{max}}$ is uniformly bounded from above. Taking the allxillary function
\begin{equation*}
Q=\log\kappa_{\mbox{max}}-\log(u-a)+\frac{A}{2}|X|^2,
\end{equation*}
where $a=\frac{1}{2}\inf_M(u)>0$ and $A>1$ is a constant to be determined later.
Assume that $X_0$ is the maximum point of $Q$. We choose a local orthonormal frame  $\{e_{1}, e_{2}, \cdots, e_{n}\}$ near $X_0$ such that
$$h_{ii}=\delta_{ij}h_{ij}, \quad  h_{11}\geq h_{22}\geq \cdots \geq h_{nn}$$
at $X_0$. Recalling that $\eta_{ii}=\sum_{k\neq i}h_{kk}$, we have
$$\eta_{11}\leq \eta_{22}\leq\cdots\leq\eta_{nn}.$$
It can follows that
$$
G^{11}\geq G^{22}\geq\cdots\geq G^{nn}, \quad F^{11}\leq F^{22}\leq\dots\leq F^{nn}.$$
We define a new function W by
$$ W=\log h_{11}-\log(u-a)+\frac{A}{2}|X|^2.$$
Since $h_{11}(X_0)=\kappa_{\mbox{max}}(X_0)$ and $h_{11}\leq\kappa_{\mbox{max}}$ near $X_0$, $W$ achieves a maximum at $X_0$. Hence
\begin{equation}\label{ht-c2-01}
0=W_i=\frac{h_{11i}}{h_{11}}-\frac{u_i}{u-a}+A \langle X, e_i\rangle
\end{equation}
and
\begin{equation}\label{ht-c2-02}
0\geq F^{ii}W_{ii}=F^{ii}(\log h_{11})_{ii}-F^{ii}(\log(u-a))_{ii}+\frac{A}{2}F^{ii}(|X|^2)_{ii}.
\end{equation}

We divide our proof in four steps.

\textbf{Step 1}:  We show that
\begin{equation}\label{ht-c2-1}
\begin{aligned}
0&\geq - \frac{2}{h_{11}} \sum_{i\geq 2} G^{1i, i1} h_{11i}^2 -\frac{F^{ii}h_{11i}^2}{h_{11}^2}\\
&+\frac{aF^{ii} h_{ii}^2}{u-a}+\frac{F^{ii}u_i^2}{(u-a)^2} +A\sum_i F^{ii}-C_0h_{11}-C_0\frac{1}{h_{11}}-AC_0,
\end{aligned}
\end{equation}
where $C_0$ depend on $\inf_{M} \rho$, $\|\rho\|_{C^1}$ and $\|f\|_{C^2}$ and satisfy $1+\sum_i \langle X, e_i\rangle^2 \leq C_0$.

We apply the similar argument in \eqref{ht-c1-04},
\begin{eqnarray}\label{ht-c2-04}
\frac{A}{2}F^{ii}(|X|^2)_{ii}=A\sum_i F^{ii} (1- h_{ii} \langle X, \nu\rangle)
=A\sum_i F^{ii}-Auf^{\frac{1}{k-l}}.
\end{eqnarray}
Using the similar argument in \eqref{ht-c2-110},  we obtain
\begin{eqnarray*}
F^{ii}h_{iik}=(d_X \tilde{f})(e_k)+h_{kk} (d_{\nu} \tilde{f})(e_k).
\end{eqnarray*}
By Gauss formula \eqref{Gauss for},  Weingarten formula \eqref{Wein for} and Codazzi formula \eqref{Codazzi}
$$u_i=h_{ii}\langle X,e_i\rangle, \quad u_{ii}=\sum_k h_{iik}\langle X, e_k\rangle-uh_{ii}^2+h_{ii}.$$
It follows that
\begin{equation}\label{ht-c2-05}
\begin{aligned}
-F^{ii}(\log(u-a))_{ii}&=-\frac{F^{ii}u_{ii}}{u-a}+\frac{F^{ii}u_i^2}{(u-a)^2}\\
&=-\frac{1}{u-a}\sum_k F^{ii}h_{iik}\langle X,e_k\rangle+\frac{uF^{ii}h_{ii}^2}{u-a}-\frac{F^{ii}h_{ii}}{u-a} +\frac{F^{ii}u_i^2}{(u-a)^2}\\
&=-\frac{1}{u-a}\sum_k F^{ii}h_{iik}\langle X, e_k\rangle + \frac{uF^{ii}h_{ii}^2}{u-a}-\frac{\tilde{f}}{u-a}+\frac{F^{ii}u_i^2}{(u-a)^2}\\
&\geq -\frac{1}{u-a}\sum_k h_{kk} (d_{\nu} \tilde{f}) (e_k)\langle X, e_k\rangle + \frac{uF^{ii}h_{ii}^2}{u-a}-\frac{\tilde{f}}{u-a}+\frac{F^{ii}u_i^2}{(u-a)^2} -C_1,
\end{aligned}
\end{equation}
where $C_1$ depend on $\inf_{M} \rho$, $\|\rho\|_{C^1}$ and $\|f\|_{C^1}$.

Differentiating \eqref{ht-eq-2} twice, we obtian
$$F^{ii}h_{ii11}=G^{ii}\eta_{ii11}\geq -G^{ij, rs} \eta_{ij1} \eta_{rs1}+ \sum_k h_{11k} (d_{\nu} \tilde{f})(e_k) -C_2h_{11}^2-C_2,$$
where $C_2$ depend on $\|f\|_{C^2}$. Applying the concavity of $G$ and Codazzi formula, we have
$$-G^{ij,rs} \eta_{ij1} \eta_{rs1} \geq -2 \sum_{i\geq2} G^{1i, i1} \eta_{1i1}^2=-2 \sum_{i\geq 2} G^{1i, i1} h_{1i1}^2=-2 \sum_{i\geq 2} G^{1i, i1} h_{11i}^2.$$
Combining with \eqref{2rd}, we have
\begin{equation}\label{ht-c2-06}
\begin{aligned}
F^{ii}(\log h_{11})_{ii}&=\frac{F^{ii}h_{11ii}}{h_{11}}-\frac{F^{ii}h_{11i}^2}{h_{11}^2}\\
&=\frac{F^{ii}}{h_{11}} \left( h_{ii11} + (h_{i1}^2-h_{ii}h_{11})h_{ii}+ (h_{ii}h_{11}-h_{i1}^2)h_{11} \right) -\frac{F^{ii}h_{11i}^2}{h_{11}^2}\\
&=\frac{F^{ii}}{h_{11}} h_{ii11}- F^{ii} h_{ii}^2+ \tilde{f} h_{11}-\frac{F^{ii}h_{11i}^2}{h_{11}^2}\\
&\geq - \frac{2}{h_{11}} \sum_{i\geq 2} G^{1i, i1} h_{11i}^2 +\frac{1}{h_{11}}
\sum_k h_{11k} (d_{\nu} \tilde{f})(e_k) -C_2\frac{1}{h_{11}}\\
&-\frac{F^{ii}h_{11i}^2}{h_{11}^2}-F^{ii} h_{ii}^2 -C_2h_{11}.
\end{aligned}
\end{equation}
Combining \eqref{ht-c2-02}, \eqref{ht-c2-04}, \eqref{ht-c2-05} and \eqref{ht-c2-06}, we have
\begin{equation}\label{ht-c2-07}
\begin{aligned}
0&\geq - \frac{2}{h_{11}} \sum_{i\geq 2} G^{1i, i1} h_{11i}^2 -\frac{F^{ii}h_{11i}^2}{h_{11}^2}\\
&+\frac{1}{h_{11}}
\sum_k h_{11k} (d_{\nu} \tilde{f})(e_k)-\frac{1}{u-a}\sum_k h_{kk} (d_{\nu} \tilde{f}) (e_k)\langle X, e_k\rangle \\
&+\frac{aF^{ii} h_{ii}^2}{u-a}+\frac{F^{ii}u_i^2}{(u-a)^2} +A\sum_i F^{ii}-C_3h_{11}-C_3\frac{1}{h_{11}}-AC_3.
\end{aligned}
\end{equation}
By \eqref{Codazzi} and \eqref{ht-c2-01},
\begin{equation}\label{ht-c2-08}
\begin{aligned}
&\frac{1}{h_{11}}
\sum_k h_{11k} (d_{\nu} \tilde{f})(e_k)-\frac{1}{u-a}\sum_k h_{kk} (d_{\nu} \tilde{f}) (e_k)\langle X, e_k\rangle \\
=&\sum_k \left( \frac{h_{11k}}{h_{11}}- \frac{u_k}{u-a} \right) (d_{\nu} \tilde{f})(e_k)\\
=&-A \sum_k  (d_{\nu} \tilde{f})(e_k) \langle X, e_k\rangle\\
\geq& -C_3A,
\end{aligned}
\end{equation}
which implies the inequality \eqref{ht-c2-1}.

\textbf{Step 2}:
There exists a positive constant $\delta$ such that
$$\frac{C_{n-1}^{k-1} [1-(n-2)\delta]^{k-1} -(n-1) \delta C_{n-1}^{k-2} [1+(n-2)\delta]^{k-2} }{C_n^l [1+(n-2)\delta]^l } >\frac{C_{n-1}^{k-1}}{2C_n^l}.$$

Let
$$A= \left( \|f\|_{C_0}^{1-\frac{1}{k-l}}  \frac{2k C_n^l}{(n-k+1) C_{n-1}^{k-1}} +1  \right) C_0.$$
We show that there exist constants $B_1>1$ depending on $n ,k, l, \delta$, $\inf_{M} \rho$, $\|\rho\|_{C^1}$ and $\|f\|_{C^2}$, such that
\begin{equation}
\frac{aF^{ii} h_{ii}^2}{2(u-a)}+\frac{A}{2}\sum_i F^{ii}\geq C_0h_{11},
\end{equation}
if  $h_{11} \geq B_1$.

Case 1: $|h_{ii}|\leq \delta h_{11}$ for all $i\geq 2$.\\
In this case we have
\begin{equation}\label{ht-c2-09}
|\eta_{11}| \leq (n-1) \delta h_{11}, \quad [1-(n-2)\delta]h_{11}\leq \eta_{22}\leq \cdots \leq \eta_{nn}\leq [1+(n-2)\delta]h_{11}.
\end{equation}
By the definition of $G^{ii}$ and $F^{ii}$, we obtain
\begin{equation}
\begin{aligned}
\sum_i F^{ii}&=(n-1)\sum_i G^{ii}\\
&= \frac{n-1}{k-l} \left(\frac{\sigma_k(\eta)}{\sigma_l(\eta)}\right)^{\frac{1}{k-l}-1} \frac{(n-k+1)\sigma_{k-1}(\eta)\sigma_l(\eta)-(n-l+1)\sigma_k(\eta)\sigma_{l-1}(\eta)}{\sigma_l^2(\eta)}\\
&\geq \frac{C_n^k}{C_n^{k-1}} \left(\frac{\sigma_k(\eta)}{\sigma_l(\eta)}\right)^{\frac{1}{k-l}-1} \left(\frac{\sigma_{k-1}(\eta)}{\sigma_l(\eta)}\right)  \\
&= \frac{C_n^k}{C_n^{k-1}} \left(\frac{\sigma_k(\eta)}{\sigma_l(\eta)}\right)^{\frac{1}{k-l}-1} \left(\frac{
\sigma_{k-1}(\eta|1)+\eta_{11}\sigma_{k-2}(\eta|1)}{\sigma_l(\eta)}\right)  \\
&\geq \frac{n-k+1}{k} f^{\frac{1}{k-l}-1} \frac{C_{n-1}^{k-1} [1-(n-2)\delta]^{k-1} -(n-1) \delta C_{n-1}^{k-2} [1+(n-2)\delta]^{k-2} }{C_n^l [1+(n-2)\delta]^l } h_{11}^{k-1-l}\\
&\geq f^{\frac{1}{k-l}-1} \frac{ (n-k+1)C_{n-1}^{k-1}}{2kC_n^l  } h_{11},
\end{aligned}
\end{equation}
 it implies that
$$C_0h_{11} \leq \frac{A}{2} \sum_i F^{ii}.$$

Case 2: $h_{22} > \delta h_{11}$ or $h_{nn} <- \delta h_{11}$.\\
In this case, we have
 \begin{equation}
\begin{aligned}
\frac{a F^{ii} h_{ii}^2}{2(u-a)}&\geq \frac{a}{2(\sup u-a)} \left(F^{22} h_{22}^2+F^{nn} h_{nn}^2\right)\\
&\geq  \frac{a\delta^2}{2(\sup u-a)} F^{22} h_{11}^2\\
&\geq  \frac{a\delta^2}{2n(\sup u-a)} \sum_i G^{ii}   h^2_{11} \\
&\geq  \left(\frac{C_n^k}{C_n^l}\right)^{\frac{1}{k-l}} \frac{a\delta^2 h_{11}}{2n(\sup u-a)}   h_{11}.
\end{aligned}
\end{equation}
Then, we have
$$\frac{a F^{ii} h_{ii}^2}{2(u-a)}\geq C_0 h_{11}$$
if
 $$h_{11} \geq \left(\left(\frac{C_n^k}{C_n^l}\right)^{\frac{1}{k-l}} \frac{a\delta^2}{2n(\sup u-a)} \right)^{-1}C_0.$$

\textbf{Step 3}:   We show that
$$|h_{ii}|\leq C_6 A,$$
 if  $h_{11} \geq B_1$, where $C_6$ is a constant depending on $n, k, l$, $\inf_{M} \rho$, $\|\rho\|_{C^1}$ and $\|f\|_{C^2}$.

Combining Step 1 and Step 2, we obtain
\begin{equation}\label{ht-c2-32}
\begin{aligned}
0&\geq - \frac{2}{h_{11}} \sum_{i\geq 2} G^{1i, i1} h_{11i}^2 -\frac{F^{ii}h_{11i}^2}{h_{11}^2}\\
&+\frac{aF^{ii} h_{ii}^2}{2(u-a)}+\frac{F^{ii}u_i^2}{(u-a)^2} + \frac{A}{2}\sum_i F^{ii}-C_0\frac{1}{h_{11}}-AC_0.
\end{aligned}
\end{equation}
Using \eqref{ht-c2-01}, the Concavity of $G$ and the Cauchy-Schwarz inequality,
\begin{equation*}
\begin{aligned}
0&\geq  -\frac{1+\epsilon}{(u-a)^2} F^{ii} u_i^2 - (1+ \frac{1}{\epsilon})A^2 F^{ii} \langle X, e_i\rangle^2\\
&+\frac{aF^{ii} h_{ii}^2}{2(u-a)}+\frac{F^{ii}u_i^2}{(u-a)^2} + \frac{A}{2}\sum_i F^{ii}-C_0\frac{1}{B_1}-AC_0\\
&\geq \left(\frac{a}{2(u-a)}-\frac{C_0 \epsilon}{(u-a)^2}\right) F^{ii}h^2_{ii}- \left((1+ \frac{1}{\epsilon})A^2C_0- \frac{A}{2}\right) \sum_i F^{ii}- 2AC_0 ,
\end{aligned}
\end{equation*}
where  we used $u_i=h_{ii} \langle X, e_i\rangle$ in the second inequality. Choosing $\epsilon=\frac{(u-a)a}{4C_0}$, then
\begin{equation}\label{ht-c2-31}
\begin{aligned}
0&\geq \frac{a}{4(u-a)} F^{ii}h^2_{ii}- \left((1+ \frac{4C_0}{(u-a)a})A^2C_0- \frac{A}{2}\right) \sum_i F^{ii}-2AC_0\\
&\geq \frac{a}{4(\sup u-a)} F^{ii}h^2_{ii}-\left((1+ \frac{4C_0}{a^2})A^2C_0- \frac{A}{2}\right) \sum_i F^{ii}-2AC_0.
\end{aligned}
\end{equation}
Note that $\sum_iF^{ii} = (n-1) \sum_iG^{ii} \geq (n-1) \left(\frac{C_n^k}{C_n^l}\right)^{\frac{1}{k-l}}$ and
$$F^{ii} \geq F^{22} \geq \frac{1}{n(n-1)} \sum_i F^{ii}.$$
Then \eqref{ht-c2-31} gives that
$$0\geq \frac{a}{4(\sup u-a)n(n-1)}  \left(\sum_{k\geq 2} h_{kk}^2\right) \sum_iF^{ii}-\left((1+ \frac{4C_0}{a^2})A^2C_0- \frac{A}{2}+ \frac{2C_0}{n-1} A \left(\frac{C_n^k}{C_n^l}\right)^{-\frac{1}{k-l}}\right) \sum_i F^{ii},$$
which implies that
$$\sum_{k\geq 2} h_{kk}^2 \leq C_6 A^2.$$

\textbf{Step 4}:
We show that there exists a constant $C$ depending on $n, k, l$, $\inf_{M} \rho$, $\|\rho\|_{C^1}$ and $\|f\|_{C^2}$ such that
$$h_{11}\leq C.$$
Without loss of generality, we assume that
\begin{equation}\label{ab}
  h_{11} \geq \max \left\{B_1, \left(\frac{32n C_0 A^2 (\sup u -a)}{\epsilon a}\right)^{\frac{1}{2}}, \frac{C_6 A}{\alpha}\right\},
\end{equation}
where $\alpha<1$ will be determined later.
Recalling $u_1=h_{11} \langle X, e_1\rangle$, by  \eqref{ht-c2-01} and the Cauchy-Schwarz inequality, we have
\begin{equation}
\begin{aligned}
\frac{F^{11} h_{111}^2}{h_{11}^2}&\leq \frac{1+\epsilon}{(u-a)^2} F^{11} u_1^2+(1+\frac{1}{\epsilon}) A^2 F^{11}\langle X, e_1\rangle^2\\
&\leq \frac{F^{11} u_1^2}{(u-a)^2} +\frac{C_0\epsilon F^{11} h_{11}^2}{(u-a)^2}  +(\frac{1+\epsilon}{\epsilon}) C_0A^2 F^{11}.
\end{aligned}
\end{equation}
We choose $\epsilon$ sufficienly small such that
$$\frac{F^{11} h_{111}^2}{h_{11}^2}\leq \frac{F^{11} u_1^2}{(u-a)^2} +\frac{a F^{ii} h_{ii}^2}{16(u-a)}  +\frac{2C_0A^2 F^{11}}{\epsilon}.$$
 Hence Combining with Step 3 and \eqref{ab}, we know that
\begin{equation}\label{ht-c2-41}
\begin{aligned}
\frac{F^{11} h_{111}^2}{h_{11}^2}\leq \frac{F^{11} u_1^2}{(u-a)^2} +\frac{a F^{ii} h_{ii}^2}{8(u-a)}
\end{aligned}
\end{equation}
and
$$|h_{ii}|\leq \alpha h_{11}, \quad \forall i\geq 2.$$
Thus
$$\frac{1}{h_{11}}\leq \frac{1+\alpha}{h_{11}-h_{ii}}.$$
Combining with Proposition \ref{th-lem-07}, we obtain
\begin{equation}\label{ht-c2-42}
\begin{aligned}
\sum_{i\geq 2} \frac{F^{ii}h_{11i}^2}{h_{11}^2}&=\sum_{i\geq 2} \frac{F^{ii}-F^{11}}{h_{11}^2} h_{11i}^2 +\sum_{i\geq 2} \frac{F^{11}h_{11i}^2}{h_{11}^2}\\
&\leq \frac{1+\alpha}{h_{11}} \sum_{i\geq 2} \frac{F^{ii}-F^{11}}{h_{11}-h_{ii}} h_{11i}^2+\sum_{i\geq 2} \frac{F^{11}h_{11i}^2}{h_{11}^2}\\
&=\frac{1+\alpha}{h_{11}} \sum_{i\geq 2} \frac{G^{11}-G^{ii}}{\eta_{ii}-\eta_{11}} h_{11i}^2+\sum_{i\geq 2} \frac{F^{11}h_{11i}^2}{h_{11}^2}\\
&=-\frac{1+\alpha}{h_{11}} \sum_{i\geq 2}G^{1i, i1} h_{11i}^2 +\sum_{i\geq 2} \frac{F^{11}h_{11i}^2}{h_{11}^2}.
\end{aligned}
\end{equation}
Using  \eqref{ht-c2-01},  Cauchy-Schwarz ineuqlity and the fact $u_i=h_{ii}\langle X, e_i \rangle$, we have
\begin{equation}\label{ht-c2-43}
\begin{aligned}
\sum_{i\geq 2} \frac{F^{11}h_{11i}^2}{h_{11}^2}&\leq 2\sum_{i\geq 2} \frac{F^{11}u_i^2}{(u-a)^2}+2A^2\sum_{i\geq 2} F^{11}  \langle X, e_i \rangle^2\\
&\leq 2\frac{C_0}{a^2} \sum_{i\geq 2} \frac{aF^{11}h_{ii}^2}{(u-a)}+2n C_0A^2 F^{11}  \\
&\leq \alpha^2 \frac{2nC_0}{a^2}  \frac{aF^{11}h_{11}^2}{(u-a)}+
\frac{\epsilon a}{16 (\sup u-a)} F^{11}h_{11}^2.
\end{aligned}
\end{equation}
We choose $\alpha$ sufficiently small such that $ \alpha \leq \min \left\{\sqrt{\frac{a^2}{32n C_0}}, 1\right\}$,  \eqref{ht-c2-42} and \eqref{ht-c2-43} implies that
\begin{equation}\label{ht-c2-44}
\begin{aligned}
\sum_{i\geq 2} \frac{F^{ii}h_{11i}^2}{h_{11}^2}
\leq-\frac{2}{h_{11}} \sum_{i\geq 2}G^{1i, i1} h_{11i}^2 +\frac{a F^{11}h_{11}^2}{8 ( u-a)}.
\end{aligned}
\end{equation}
Substituting  \eqref{ht-c2-41}  and  \eqref{ht-c2-44} into \eqref{ht-c2-32}, we obtian that
\begin{equation}
\begin{aligned}
0&\geq \frac{a F^{ii}h_{ii}^2}{4 ( u-a)}+ \frac{A}{2} \sum_i F^{ii} -C_0(A+1)\\
&\geq \frac{C_0}{2}h_{11}-C_0(A+1),
\end{aligned}
\end{equation}
which implies that
$$h_{11}\leq 2(A+1).$$
\end{proof}


\section{The proof of  Theorem  \ref{Main}}

In this section, we use the degree theory for nonlinear elliptic
equation developed in \cite{Li89} to prove Theorem \ref{Main}. The
proof here is similar to \cite{An, Jin, Li-Sh}. So, only sketch will
be given below.

After establishing the  priori estimates in Theorem \ref{ht-C0},
Theorem \ref{ht-C1e} and Theorem \ref{ht-c2-1}, we know that the
equation \eqref{ht-Eq} is uniformly elliptic. From \cite{Eva82},
\cite{Kry83}, and Schauder estimates, we have
\begin{eqnarray}\label{C2+}
|\rho|_{C^{4,\alpha}(\mathbb{S}^n)}\leq C
\end{eqnarray}
for any $(\eta, k)$-convex solution $M$ to the equation \eqref{ht-Eq}, where
the position vector of $M$ is $X=\rho(x)x$ for $x \in \mathbb{S}^n$.
We define
\begin{eqnarray*}
C_{0}^{4,\alpha}(\mathbb{S}^n)=\{\rho \in
C^{4,\alpha}(\mathbb{S}^n): M \ \mbox{is}
 \ (\eta, k)-\mbox{convex}\}.
\end{eqnarray*}
Let us consider $$F(.; t): C_{0}^{4,\alpha}(\mathbb{S}^n)\rightarrow
C^{2,\alpha}(\mathbb{S}^n),$$ which is defined by
\begin{eqnarray*}
F(\rho, x; t)=\frac{\sigma_k(\lambda(\eta))}{\sigma_l(\lambda(\eta))}-f^t(X, \nu),
\end{eqnarray*}
where
\begin{eqnarray*}
f^t(X, \nu)=tf(X, \nu)+ (1-t)  \frac{C_n^k}{C_n^l}  (n-1)^{k-l} \left( \frac{1}{|X|^{k-l}} +\epsilon (\frac{1}{|X|^{k-l}}-1)\right),
\end{eqnarray*}
where the constant $\epsilon$ is small sufficiently such that
\begin{eqnarray*}
\min_{r_1\leq \rho\leq r_2}\left( \frac{1}{\rho^{k-l}} +\epsilon (\frac{1}{\rho^{k-l}}-1)\right)\geq c_0>0
\end{eqnarray*}
for some positive constant $c_0$.
Let $$\mathcal{O}_R=\{\rho \in C_{0}^{4,\alpha}(\mathbb{S}^n):
|\rho|_{C^{4,\alpha}(\mathbb{S}^n)}<R\},$$ which clearly is an open
set of $C_{0}^{4,\alpha}(\mathbb{S}^n)$. Moreover, if $R$ is
sufficiently large, $F(\rho, x; t)=0$ has no solution on $\partial
\mathcal{O}_R$ by the prior estimate established in \eqref{C2+}.
Therefore the degree $\deg(F(.; t), \mathcal{O}_R, 0)$ is
well-defined for $0\leq t\leq 1$. Using the homotopic invariance of
the degree, we have
\begin{eqnarray*}
\deg(F(.; 1), \mathcal{O}_R, 0)=\deg(F(.; 0), \mathcal{O}_R, 0).
\end{eqnarray*}
Theorem \ref{Uni} shows that $\rho_0=1$ is the unique
solution to the above equation for $t=0$. Direct calculation show
that
\begin{eqnarray*}
F(s, x; 0)= -\epsilon \frac{C_n^k}{C_n^l}  (n-1)^{k-l} \left( \frac{1}{s^{k-l}} -1\right).
\end{eqnarray*}
Then
\begin{eqnarray*}
\delta_{\rho_0}F(\rho_0, x; 0)=\frac{d}{d s}|_{s=1}F(s\rho_0, x;
0)=\epsilon \frac{C_n^k}{C_n^l}  (n-1)^{k-l}(k-l)>0,
\end{eqnarray*}
where $\delta F(\rho_0, x; 0)$ is the linearized operator of $F$ at
$\rho_0$. Clearly, $\delta F(\rho_0, x; 0)$ takes the form
\begin{eqnarray*}
\delta_{w}F(\rho_0, x; 0)=-a^{ij}w_{ij}+b^i
w_i+\epsilon \frac{C_n^k}{C_n^l}  (n-1)^{k-l}(k-l),
\end{eqnarray*}
where $a^{ij}$ is a positive definite matrix. Since
$\epsilon \frac{C_n^k}{C_n^l}  (n-1)^{k-l}(k-l)>0,$
thus $\delta_{\rho_0} F(\rho_0, x; 0)$ is an invertible operator. Therefore,
\begin{eqnarray*}
\deg(F(.; 1), \mathcal{O}_R; 0)=\deg(F(.; 0), \mathcal{O}_R, 0)=\pm
1.
\end{eqnarray*}
So, we obtain a solution at $t=1$. This completes the proof of
Theorem \ref{Main}.

\end{document}